\documentclass[12pt]{article}
\usepackage{amsmath,amssymb,amsthm}
\newcommand{\R}{{\mathbb R}}
\newcommand{\Z}{{\mathbb Z}}
\newcommand{\N}{{\mathbb N}}

\newcommand{\Q}{{\mathbb Q}}
\newcommand{\D}{{\mathcal D}}

\newcommand{\B}{{\mathcal B}}

\newcommand{\esslim}{\mathop{\rm ess\,lim}}

\newcommand{\sign}{\mathop{\rm sign}}
\newcommand{\supp}{\mathop{\rm supp}}
\renewcommand{\div}{{\rm div}}

\renewcommand{\>}{\rangle}

\def\Xint#1{\mathchoice
{\XXint\displaystyle\textstyle{#1}}%
{\XXint\textstyle\scriptstyle{#1}}%
{\XXint\scriptstyle\scriptscriptstyle{#1}}%
{\XXint\scriptscriptstyle\scriptscriptstyle{#1}}%
\!\int}
\def\XXint#1#2#3{{\setbox0=\hbox{$#1{#2#3}{\int}$ }
\vcenter{\hbox{$#2#3$ }}\kern-.57\wd0}}

\def\dashint{\Xint-}
\numberwithin{equation}{section}

\theoremstyle{plain}
\newtheorem{theorem}{Theorem}[section]
\newtheorem{lemma}{Lemma}[section]

\newtheorem{corollary}{Corollary}[section]

\theoremstyle{definition}
\newtheorem{definition}{Definition}[section]
\newtheorem{remark}{Remark}[section]

\date\empty

\title{On the Cauchy problem for scalar conservation laws on the Bohr compactification of $\R^n$}
\author{E.Yu.~Panov\footnote{Novgorod State University, e-mail: Eugeny.Panov@novsu.ru}}

\begin{document}

\maketitle
\begin{abstract} We study the Cauchy problem for a multidimensional scalar conservation law on the Bohr compactification of $\R^n$. The existence and uniqueness of entropy solutions are established in the general case of merely continuous flux vector. We propose also the necessary and sufficient condition for the decay of entropy solutions as time $t\to+\infty$.
\end{abstract}

\section{Introduction.}

Let $AP(\R^n)$ be the algebra of Bohr almost periodic functions. These functions can be described as uniform limits of trigonometric polynomials on $\R^n$ (~i.e., finite sums $\sum a_\lambda e^{2\pi i\lambda\cdot x}$, with ${i^2=-1}$, $\lambda\in\R^n$~). Denote by $C_R$ the cube
$$\{ \ x=(x_1,\ldots,x_n)\in\R^n \ | \ |x|_\infty=\max_{j=1,\ldots,n}|x_j|\le R/2 \ \}, \quad R>0$$ and let
$$N_1(u)=\limsup_{R\to +\infty} R^{-n}\int_{C_R} |u(x)|dx
$$
be the mean $L^1$-norm of a function $u(x)\in L^1_{loc}(\R^n)$.
Recall (~see \cite{Bes,Lev}~) that Besicovitch space $\B^1(\R^n)$ is the closure of trigonometric polynomials in the quotient space $B^1(\R^n)/B^1_0(\R^n)$, where
$$
B^1(\R^n)=\{ u\in L^1_{loc}(\R^n) \ | \ N_1(u)<+\infty \}, \ B^1_0(\R^n)=\{ u\in L^1_{loc}(\R^n) \ | \ N_1(u)=0 \}.
$$
The space $\B^1(\R^n)$ is equipped with the norm $\|u\|_1=N_1(u)$ (we identify classes in the quotient space and their representatives).
The space $\B^1(\R^n)$ is a Banach space, it is isomorphic to the completeness of the space $AP(\R^n)$ of Bohr almost periodic functions with respect to the norm $N_1$.

It is known \cite{Bes} that for each $u\in \B^1(\R^n)$ there exist the mean value
$$\dashint_{\R^n} u(x)dx\doteq\lim\limits_{R\to+\infty}R^{-n}\int_{C_R} u(x)dx$$ and, more generally, the Bohr-Fourier coefficients
$$
a_\lambda=\dashint_{\R^n} u(x)e^{-2\pi i\lambda\cdot x}dx, \quad\lambda\in\R^n.
$$
The set
$$ Sp(u)=\{ \ \lambda\in\R^n \ | \ a_\lambda\not=0 \ \} $$ is called the spectrum of an almost periodic function $u$.
It is known \cite{Bes} that the spectrum $Sp(u)$ is at most countable. Denote by $M(u)$ the smallest additive subgroup of $\R^n$ containing $Sp(u)$ (notice that $M(u)$ is always countable whenever it is different from the zero subgroup).

Now we consider the Cauchy problem for the conservation law
\begin{equation}\label{1}
u_t+\div_x\varphi(u)=0
\end{equation}
with initial data
\begin{equation}\label{ini}
u(0,x)=u_0(x)\in L^\infty(\R^n).
\end{equation}
The flux vector $\varphi(u)$ is supposed to be only continuous: $$\varphi(u)=(\varphi_1(u),\ldots,\varphi_n(u))\in C(\R,\R^n).$$
Recall the notion of entropy solution of (\ref{1}), (\ref{ini}) in the sense of S.N.~Kruzhkov \cite{Kr}.

\begin{definition}\label{def1}
A bounded measurable function $u=u(t,x)\in L^\infty(\Pi)$ is called an entropy solution (e.s. for
short) of (\ref{1}), (\ref{ini}) if for all $k\in\R$
\begin{equation}\label{2}
|u-k|_t+\div_x[\sign(u-k)(\varphi(u)-\varphi(k))]\le 0
\end{equation}
in the sense of distributions on $\Pi$ (in $\D'(\Pi)$);
$$\esslim_{t\to 0} u(t,\cdot)=u_0 \ \mbox{ in } L^1_{loc}(\R^n).$$
\end{definition}

Condition (\ref{2}) means that for all non-negative test functions $f=f(t,x)\in C_0^1(\Pi)$
$$
\int_\Pi [|u-k|f_t+\sign(u-k)(\varphi(u)-\varphi(k))\cdot\nabla_xf]dtdx\ge 0
$$
(here $\cdot$ denotes the inner product in $\R^n$).

It is known that e.s. always exists (see \cite{KrPa1,PaMV,PaIzv}~) but, in the case under consideration when the flux functions are merely continuous, this e.s. may be nonunique (see examples in \cite{KrPa1,KrPa2}).

In recent preprint \cite{prep} problem (\ref{1}), (\ref{ini}) was studied in the case when $u_0\in\B^1(\R^n)\cap L^\infty(\R^n)$. The following results were established.

\begin{theorem}\label{th1}
Let $u(t,x)$ be an e.s. of problem (\ref{1}), (\ref{ini}).
Then, after possible correction on a set of null measure, ${u(t,\cdot)\in C([0,+\infty),\B^1(\R^n))\cap L^\infty(\Pi)}$ and for all $t>0$  $M(u(t,\cdot))\subset M(u_0)$.
\end{theorem}

\begin{theorem}[decay property]\label{th2}
Assume that
\begin{eqnarray}\label{ND}
\forall\xi\in M(u_0), \xi\not=0 \ \mbox{ the functions } u\to\xi\cdot\varphi(u) \nonumber\\
\mbox{ are not affine on non-empty intervals }
\end{eqnarray}
(the linear non-degeneracy condition). Let $u(t,x)$ be an e.s. of problem (\ref{1}), (\ref{ini}). Then
\begin{equation}\label{dec}
\lim_{t\to +\infty} \dashint_{\R^n}|u(t,x)-C|dx=0, \quad \mbox{ where } \quad C=\dashint_{\R^n} u_0(x)dx.
\end{equation}
Moreover, condition (\ref{ND}) is exact: if it fails, then there exists an initial function $u_0\in\B^1(\R^n)\cap L^\infty(\R^n)$ such that $Sp(u_0)\subset M(u_0)$ and the e.s. of (\ref{1}), (\ref{ini}) does not satisfy the decay property (\ref{dec}).
\end{theorem}

\section{The Bohr compactification of $\R^n$. Some auxiliary lemmas.}

In the present paper we look at the problem (\ref{1}), (\ref{ini}) from another point. Namely, we will consider this problem for functions $u(t,x)$ defined for $x\in\B_n$, where $\B_n$ is the Bohr compactification of $\R^n$. This is a compact group, which can be identified with the spectrum of the algebra $AP(\R^n)$.
There is a continuous homomorphism of the groups $in:\R^n\to\B_n$ uniquely determined by the identity $\hat f(in(x))=f(x)$ for all $f\in AP(\R^n)$, where $f\to \hat f$ is the Gelfand transform.
It is known that the homomorphism $in$ has a null kernel (i.e., it is an embedding) and its image $in(\R^n)$ is dense in $\B_n$. Denote by $m$ the Haar measure on $\B_n$. This measure represents the mean value functional, that is,
for every almost periodic function $v(x)\in AP(\R^n)$
\begin{equation}\label{ad0}
\int_{\B_n} \hat v(x)dm(x)=\dashint_{\R^n} v(x)dx.
\end{equation}
In particular,
$$
\int_{\B_n} |\hat v(x)|dm(x)=\int_{\B_n} \widehat{|v(x)|} dm(x)=\dashint_{\R^n} |v(x)|dx.
$$
It follows from this identity that the Gelfand transform admits extension to an isomorphism $\B^1(\R^n)\mathop{\to} L^1(\B_n,m)$. We keep the notation $u\to \hat u$
and the name of Gelfand transform for such isomorphism.
It turns out that under the Gelfand transform the space $\B^1(\R^n)\cap L^\infty(\R^n)$ corresponds to the space $L^\infty(\B_n,m)$.

\begin{lemma}\label{lemad1}
An almost periodic function $u(x)$ belongs to the space $\B^1(\R^n)\cap L^\infty(\R^n)$ if and only if $\hat u(x)\in L^\infty(\B_n,m)$.
\end{lemma}

\begin{proof}
It follows from the property of Gelfand transform that $\widehat{h(u)}=h(\hat u)$ for every $u=u(x)\in AP(\R^n)$, $h=h(u)\in C(\R)$. If $u(x)\in \B^1(\R^n)$ then we can find a sequence $u_r\in AP(\R^n)$ such that
$u_r\to u$ in $\B^1(\R^n)$ as $r\to\infty$ (for instance we can take the Bochner-Fej\'er approximations, see \cite{Bes,Lev}). Since
$$
\|\widehat{u_r}-\widehat{u}\|_{L^1(\B_n,m)}=\|u_r-u\|_{\B^1(\R^n)}\mathop{\to}_{r\to\infty} 0,
$$
then $\widehat{u_r}\to \widehat{u}$ in $L^1(\B_n,m)$ as $r\to\infty$. If $h(u)$ is globally Lipschitz continuous, then $h(u_r)\to h(u)$ in
$\B^1(\R^n)$, $h(\widehat{u_r})\to h(\widehat{u})$ in $L^1(\B_n,m)$ as $r\to\infty$. The former relation implies that $h(\widehat{u_r})=\widehat{h(u_r)}\to\widehat{h(u)}$ in $L^1(\B_n,m)$ as $r\to\infty$, and we claim that
$\widehat{h(u)}=h(\widehat{u})$. In particular, it follows from (\ref{ad0}) that
$$
\dashint_{\R^n} h(u(x))dx=\int_{\B_n} h(\widehat{u}(x))dm(x).
$$
Taking in this relation $h(u)=(|u|-M)^+\doteq\max(0,|u|-M)$, $M\ge 0$, we arrive at
\begin{equation}\label{pres}
\dashint_{\R^n} (|u(x)|-M)^+dx=\int_{\B_n} (|\hat u(x)|-M)^+ dm(x).
\end{equation}
If $u(x)\in\B^1(\R^n)\cap L^\infty(\R^n)$, and $M=\|u\|_\infty$ then it follows from (\ref{pres}) that $(|\hat u(x)|-M)^+=0$, that is, $|\hat u(x)|\le M$ $m$-a.e. in $\B_n$. This means that
$\hat u\in L^\infty(\B_n,m)$, $\|\hat u\|_\infty\le M$. Conversely, if $\hat u\in L^\infty(\B_n,m)$, $M=\|\hat u\|_\infty$, then ${\displaystyle\dashint_{\R^n} (|u(x)|-M)^+dx=0}$. We introduce the function
$v(x)=\max(-M,\min(M,u(x)))\in\B^1(\R^n)\cap L^\infty(\R^n)$. Evidently, $\displaystyle N_1(u-v)=\dashint_{\R^n} (|u(x)|-M)^+dx=0$. Therefore, $u=v$ in $\B^1(\R^n)$,
and we conclude that $u\in\B^1(\R^n)\cap L^\infty(\R^n)$, as required.
\end{proof}

Remark that, as one can realize from the proof of Lemma~\ref{lemad1},
$$\|\hat u\|_\infty=\inf\{ \ \|v\|_\infty \ | \ v\in L^\infty(\R^n), \ v=u \mbox{ in } \B^1(\R^n) \ \}.$$

Let $\rho(s)\in C_0(\R)$ be a nonnegative function such that ${\displaystyle\int_{-\infty}^{+\infty}\rho(s)ds=1}$. We introduce the approximate unity $\delta_\nu(s)=\nu\rho(\nu s)$, $\nu\in\N$. If $v(t)\in L^\infty(\R)$, then almost all $t\in\R$ are Lebesgue points of $v(t)$, which implies that for such $t$
$\displaystyle\lim_{\nu\to\infty}\int_{\R} |v(t)-v(s)|\delta_\nu(t-s)ds=0$. Integrating this relation over $t\in\R$
(with the help of Lebesgue dominated convergence theorem), we obtain that
\begin{equation}\label{Leb1}
\lim_{\nu\to\infty}\int_{\R^2} |v(t)-v(s)|\chi(t)\delta_\nu(t-s)dtds=0
\end{equation}
for every function $\chi(t)\in L^1(\R)$.
Moreover, if $\omega(r)$ is a continuous nonnegative function on $[0,+\infty)$ such that $\omega(0)=0$, then
\begin{equation}\label{Leb2}
\lim_{\nu\to\infty}\int_{\R^2} \omega(|v(t)-v(s)|)\chi(t)\delta_\nu(t-s)dtds=0,
\end{equation}
cf. the proof of Corollary~\ref{corad2} below.

We need also the following well-known property of almost periodic functions.
\begin{lemma}\label{lemad2}
Let $v(x)\in\B^1(\R^n)$ and $g(y)\in C_0(\R^n)$. Then
\begin{equation}\label{aver}
\lim_{R\to+\infty} R^{-n}\int_{\R^n} v(x)g(x/R)dx=C\dashint_{\R^n} v(x)dx,
\end{equation}
where $\displaystyle C=\int_{\R^n} g(y)dy$.
\end{lemma}

\begin{proof}
Let $k>0$ be so large that $\supp g$ is included in the cube $C_k$, and $M=\|g\|_\infty$. Then
\begin{eqnarray*}
\limsup_{R\to+\infty}R^{-n}\left|\int_{\R^n} v(x)g(x/R)dx\right|\le \\ Mk^n \limsup_{R\to+\infty}(kR)^{-n}\int_{C_{kR}}|v(x)|dx=Mk^n\|v\|_{\B^1(\R^n)}.
\end{eqnarray*}
The above relation implies that both parts of equality (\ref{aver}) are continuous with respect to $v\in\B^1(\R^n)$.
Since trigonometric polynomials are dense in $\B^1(\R^n)$, it is sufficient to prove (\ref{aver})
for functions $v(x)=e^{2\pi i\lambda\cdot x}$, $\lambda\in\R^n$. For such functions, making the change $y=x/R$, we obtain
\begin{eqnarray*}
R^{-n}\int_{\R^n} v(x)g(x/R)dx=\int_{\R^n} e^{2\pi iR\lambda\cdot y}g(y)dy\mathop{\to}_{R\to+\infty}
\left\{\begin{array}{lr} 0, & \lambda\not=0, \\ C, & \lambda=0 \end{array}\right. = C\dashint_{\R^n} v(x)dx,
\end{eqnarray*}
since $\displaystyle e^{2\pi iR\lambda\cdot y}\mathop{\rightharpoonup}_{R\to+\infty} 0$ weakly-$*$ in $L^\infty(\R^n)$ if $\lambda\not=0$. The proof is complete.
\end{proof}

Let $\Lambda=\{\lambda_j\}_{j=1}^N$ be an at most countable subset of $\R^n$, consisting of vectors independent over the field of rationales $\Q$. Here $N\in\N\cup\{\infty\}$. We define the corresponding sequence of Bochner-Fej\'er kernels
\begin{eqnarray*}
\Phi_r(x)=\sum_{\bar k\in\Z^{N_r},|\bar k|_\infty< (r+1)!}\prod_{j=1}^{N_r}\left(1-\frac{|k_j|}{(r+1)!}\right)e^{\frac{2\pi i}{r!}\sum_{j=1}^{N_r}k_j\lambda_j\cdot x}=\\
\frac{1}{((r+1)!)^{N_r}}\prod_{j=1}^{N_r}\frac{\sin^2 (\pi (r+1)\lambda_j\cdot x)}{\sin^2 (\pi\lambda_j\cdot x/r!)},
\end{eqnarray*}
where $N_r=N$ if $N<\infty$, $N_r=r$, otherwise.

\begin{lemma}\label{lemad3}
Let $v(x)\in\B^1(\R^n)$, and $\Phi_r(x)$, $r\in\N$ be the sequence
of Bochner-Fej\'er kernels corresponding to the basis $\Lambda$ of $\Q$-linear subspace $H$ of $\R^n$ generated by $Sp(v)$. Then
\begin{equation}\label{ad1}
\lim_{r\to\infty}\dashint_{\R^n}\left(\dashint_{\R^n}|v(x)-v(y)|\Phi_r(x-y)dy\right)dx=0.
\end{equation}
\end{lemma}

\begin{proof}
For a positive $\varepsilon$ we can choose a trigonometric polynomial $\displaystyle w(x)=\sum_{\lambda\in\Lambda} a_\lambda e^{2\pi i\lambda\cdot x}$
such that $\Lambda=Sp(w)\subset Sp(v)$ and $$\|v-w\|_{\B^1(\R^n)}=\dashint_{\R^n}|v(x)-w(x)|dx<\varepsilon/2.$$ Since, evidently,
$$
||v(x)-v(y)|-|w(x)-w(y)||\le |v(x)-w(x)|+|v(y)-w(y)|
$$
while $\Phi_r\ge 0$, then
\begin{eqnarray}\label{ad2}
\left|\dashint_{\R^n}\left(\dashint_{\R^n}|v(x)-v(y)|\Phi_r(x-y)dy\right)dx-\right. \nonumber\\ \left. \dashint_{\R^n}\left(\dashint_{\R^n}|w(x)-w(y)|\Phi_r(x-y)dy\right)dx\right|\le\nonumber\\
\dashint_{\R^n}\left(\dashint_{\R^n}|v(x)-w(x)|\Phi_r(x-y)dy\right)dx+ \nonumber\\ \dashint_{\R^n}\left(\dashint_{\R^n}|v(y)-w(y)|\Phi_r(x-y)dx\right)dy=\nonumber\\
\dashint_{\R^n}|v(x)-w(x)|dx+\dashint_{\R^n}|v(y)-w(y)|dy=2\dashint_{\R^n}|v(x)-w(x)|dx<\varepsilon.
\end{eqnarray}
We use here that $\displaystyle\dashint_{\R^n}\Phi_r(z)dz=1$ for all $r\in\N$, and that
\begin{equation}\label{comm}
\dashint_{\R^n}\left(\dashint_{\R^n}G(x,y)dx\right)dy=\dashint_{\R^n}\left(\dashint_{\R^n}G(x,y)dy\right)dx=
\dashint_{\R^{2n}}G(x,y)dxdy,
\end{equation}
where $G(x,y)=|v(y)-w(y)|\Phi_r(x-y)$. Indeed, if $v(y)\in AP(\R^n)$ then $G(x,y)\in AP(\R^{2n})$ and (\ref{comm}) follows from results of \cite[Ch. I, \S~12]{Bes}. The general case $v(y)\in\B^1(\R^n)$ is treated with the help of approximation of $v(y)$ by Bohr almost periodic functions $v_r\in AP(\R^n)$, $r\in\N$, and passage to the limit as $r\to\infty$. Observe also that
\begin{eqnarray}\label{ad3}
\dashint_{\R^n}|w(x)-w(y)|\Phi_r(x-y)dy\le \sum_{\lambda\in\Lambda} |a_\lambda|\dashint_{\R^n}|e^{2\pi i\lambda\cdot x}-e^{2\pi i\lambda\cdot y}|\Phi_r(x-y)dy=\nonumber\\
\sum_{\lambda\in\Lambda} |a_\lambda|\dashint_{\R^n}|e^{2\pi i\lambda\cdot (x-y)}-1|\Phi_r(x-y)dy\doteq I_r.
\end{eqnarray}
Since the function $h_\lambda(z)=|e^{2\pi i\lambda\cdot z}-1|$ is a Bohr almost periodic function and its spectrum lays in $H$, then
$$
\dashint_{\R^n}|e^{2\pi i\lambda\cdot z}-1|\Phi_r(z)dz=\dashint_{\R^n}h_\lambda(z)\Phi_r(z)dz\mathop{\to}_{r\to\infty} h_\lambda(0)=0.
$$
This together with finiteness of $\Lambda$ implies that $I_r\to 0$ as $r\to\infty$. By (\ref{ad3}) the sequence
$$\dashint_{\R^n}|w(x)-w(y)|\Phi_r(x-y)dy \mathop{\to}_{r\to\infty} 0
$$
uniformly in $x$. This implies the relation
$$
\lim_{r\to\infty}\dashint_{\R^n}\left(\dashint_{\R^n}|w(x)-w(y)|\Phi_r(x-y)dy\right)dx=0.
$$
From this relation and (\ref{ad2}) it follows that
$$
\limsup_{r\to\infty}\dashint_{\R^n}\left(\dashint_{\R^n}|v(x)-v(y)|\Phi_r(x-y)dy\right)dx\le\varepsilon.
$$
Since $\varepsilon>0$ is arbitrary, then (\ref{ad1}) follows. The proof is complete.
\end{proof}

\begin{corollary}\label{corad1}
Let $\B_n$ be the Bohr compactification of $\R^n$, $\widehat{\Phi_r}(x)$, $r\in\N$, be the extension of Bochner-Fej\'er kernels on $\B_n$ (i.e., the Gelfand transform of $\Phi_r$). Then for each $v(x)\in L^1(\B_n,m)$
\begin{equation}\label{ad4}
\lim_{r\to\infty}\int_{\B_n\times\B_n}|v(x)-v(y)|\widehat{\Phi_r}(x-y)dm(x)dm(y)=0.
\end{equation}
\end{corollary}

\begin{proof}
There exists a unique function $u(x)\in\B^1(\R^n)$ such that $v=\hat u$. Since for each fixed $x\in\R^n$
$$
(|u(x)-u(\cdot)|\Phi_r(x-\cdot))^{\wedge}=|u(x)-v(y)|\widehat{\Phi_r}(x-y),
$$
we have
\begin{equation}\label{ad5}
F(x)\doteq\dashint_{\R^n}|u(x)-u(y)|\Phi_r(x-y)dy=\int_{\B_n}|u(x)-v(y)|\widehat{\Phi_r}(x-y)dm(y).
\end{equation}
As can be easily verified, for all $x\in\B_n$
$$
\hat F(x)=\int_{\B_n}|v(x)-v(y)|\widehat{\Phi_r}(x-y)dm(y).
$$
This equality and (\ref{ad5}) imply the relation
\begin{eqnarray}\label{ad6}
\dashint_{\R^n}\left(\dashint_{\R^n}|u(x)-u(y)|\Phi_r(x-y)dy\right)dx=\dashint_{\R^n} F(x)dx=\nonumber\\
\int_{\B_n} \hat F(x) dm(x)=
\int_{\B_n}\left(\int_{\B_n}|v(x)-v(y)|\widehat{\Phi_r}(x-y)dm(y)\right)dm(x)=\nonumber\\ \int_{\B_n\times\B_n}|v(x)-v(y)|\widehat{\Phi_r}(x-y)dm(x)dm(y).
\end{eqnarray}
Now relation (\ref{ad4}) follows from (\ref{ad6}) and the statement of Lemma~\ref{lemad3}.
\end{proof}
\begin{corollary}\label{corad2}
Let $\omega(r)$ be a continuous function on $[0,+\infty)$ such that ${\omega(r)\ge\omega(0)=0}$, and
$v(x)\in L^\infty(\B_n,m)$. Then
\begin{equation}\label{ad6a}
\lim_{r\to\infty}\int_{\B_n\times\B_n}\omega(|v(x)-v(y)|)\widehat{\Phi_r}(x-y)dm(x)dm(y)=0.
\end{equation}
\end{corollary}

\begin{proof}
Let $M=\|v\|_\infty$ and $\varepsilon>0$. Then there exists a constant $C>0$ such that
\begin{equation}\label{ad6b}
\omega(r)\le\varepsilon+Cr \quad \forall r\in [0,2M].
\end{equation}
Indeed, we can find a positive $\delta$ such that $\omega(r)<\varepsilon$ for $0\le r<\delta$. Taking $\displaystyle C=\frac{1}{\delta}\max_{\delta\le r\le 2M}\omega(r)$, we see that (\ref{ad6b}) is satisfied. In view of (\ref{ad6b}) and Corollary~\ref{corad1} we obtain the relation
\begin{eqnarray*}
\limsup_{r\to\infty}\int_{\B_n\times\B_n}\omega(|v(x)-v(y)|)\widehat{\Phi_r}(x-y)dm(x)dm(y)\le \\
\varepsilon+C\lim_{r\to\infty}\int_{\B_n\times\B_n}|v(x)-v(y)|\widehat{\Phi_r}(x-y)dm(x)dm(y)=\varepsilon.
\end{eqnarray*}
Since $\varepsilon>0$ is arbitrary, we conclude that (\ref{ad6a}) holds.
\end{proof}

\section{Setting of the problem. The main results.}

Let us look at the problem (\ref{1}), (\ref{ini}) in the framework of distributions on the space $\R_+\times\B_n$.
Since $\R^n$ can be identified with a subgroup of $\B_n$ we can define partial derivatives $f_{x_j}(y)$, $j=1,\ldots,n$, of a function $f(y)$ on $\B_n$ as partial derivatives $h_{x_j}(0)$ of the function $h(x)=f(x+y)$, $x\in\R^n$. This allows to introduce the spaces $C^k(\B_n)$, $k=0,1,\ldots$ and the space $C^\infty(\B_n)$. As usual, distributions on $\B_n$ are the linear continuous functionals on $C^\infty(\B_n)$. In the same way as in the classic situation we can introduce generalized derivatives $f_{x_j}$ of a function $f(y)\in L^1(\B_n,m)$. They are functionals
$$
\<f_{x_j},g\>=-\<f,g_{x_j}\>=-\int_{\B_n}f(y)g_{x_j}(y)dm(y)
$$
acting on the space of test functions $g(y)\in C^1(\B_n)$.

Hence, we may study the Cauchy problem:
\begin{equation}\label{adeq}
v_t+\div_x\varphi(v)=v_t+\sum_{j=1}^n (\varphi_j(v))_{x_j}=0,
\end{equation}
$v=v(t,y)$, $t>0$, $y\in\B_n$, with initial data
\begin{equation}\label{adini}
v(0,y)=v_0(y)\in L^\infty(\B_n,m).
\end{equation}
We denote by $C^1_0(\R_+\times\B_n)$ the space of compactly supported test functions $f=f(t,x)$, which have continuous partial derivatives $f_t$, $f_{x_j}$, $j=1,\ldots,n$, defined in the usual way:
$$
f_t(t,y)=\lim_{\delta\to 0} \frac{f(t+\delta,y)-f(t,y)}{\delta}, \ f_{x_j}(t,y)=\lim_{\delta\to 0} \frac{f(t,y+\delta e_j)-f(t,y)}{\delta},
$$
where $t>0$, $y\in\B_n$, and $e_j$, $j=1,\ldots,n$, being the canonical basis in $\R^n$. It is clear that these derivatives coincides with classic derivatives of the function $h(t,x)=f(t,y+x)$ at the points $(t,0)$. Actually, the space $C^1_0(\R_+\times\B_n)$ consists of functions $f(t,y)=\hat h(t,y)$ being the Gelfand transforms (as functions of spacial variables) of the functions $h(t,x)\in C^1(\Pi)$, which together with their partial derivatives are Bohr almost periodic functions (with respect to the spacial variables), and have supports inside layers $t\in [a,b]$, $b>a>0$ (where $a,b$ depend on $h$). Besides, as is easy to verify, $f_t=\widehat{h_t}$, $f_{x_j}=\widehat{h_{x_j}}$, $j=1,\ldots,n$.

\begin{definition}\label{def2}
A function $v=v(t,y)\in L^\infty(\R_+\times\B_n)=L^\infty(\R_+\times\B_n,dt\times m)$ is called an e.s. of problem (\ref{adeq}), (\ref{adini}) if for all $k\in\R$
$$
|v-k|_t+\div_x [\sign(v-k)(\varphi(v)-\varphi(k))]\le 0
$$
in the sense of distributions on $\R_+\times\B_n$, that is,
\begin{equation}\label{ad7}
\int_{\R_+\times\B_n} [|v-k|f_t+\sign(v-k)(\varphi(v)-\varphi(k))\cdot\nabla_x^y f]dtdm(y)\ge 0
\end{equation}
for every nonnegative test function $f=f(t,y)\in C^1_0(\R_+\times\B_n)$, and
$$
\esslim_{t\to 0} v(t,\cdot)=v_0 \ \mbox{ in } L^1(\B_n,m).
$$
\end{definition}
In (\ref{ad7}) we denote by $\nabla_x^y f$ the vector with coordinates $f_{x_j}(t,y)$, $j=1,\ldots,n$.
\begin{theorem}\label{adth1}
There exists a unique e.s. of (\ref{adeq}), (\ref{adini}). Moreover,
$v(t,y)=\hat u(t,y)$, where $u(t,x)\in C([0,+\infty),\B^1(\R^n))\cap L^\infty(\Pi)$ be an almost periodic e.s. to the problem (\ref{1}), (\ref{ini}) with the initial functions $u_0=u_0(x)$ such that $v_0=\widehat{u_0}$.
\end{theorem}

\begin{proof}
By Lemma~\ref{lemad1} there exists a unique $u_0=u_0(x)\in\B^1(\R^n)\cap L^\infty(\R^n)$ such that $v_0=\widehat{u_0}$.
Let $u(t,x)\in C([0,+\infty),\B^1(\R^n))\cap L^\infty(\Pi)$ be an almost periodic e.s. to the problem (\ref{1}), (\ref{ini}) with the initial functions $u_0(x)$. We choose a function $g(y)\in C_0^1(\R^n)$, $g(y)\ge 0$,
$g\not\equiv 0$, and apply relation (\ref{2}) to the test function $\displaystyle f_R=\frac{1}{CR^n}h(t,x)g(x/R)$, where
$\displaystyle C=\int_{\R^n} g(y)dy$, $R>0$, and $h(t,x)\in C^1(\Pi)$ is a nonnegative function, which is almost periodic (with respect to the spacial variables) together with its partial derivatives, and is supported in some layer $0<a<t<b$. We obtain that for each $k\in\R$
\begin{eqnarray}\label{ad8}
\frac{1}{CR^n}\int_{\Pi}[|u-k|h_t+\sign(u-k)(\varphi(u)-\varphi(k))\cdot\nabla_x h]g(x/R)dtdx+\nonumber\\
\frac{1}{CR^{n+1}}\int_{\Pi}\sign(u-k)(\varphi(u)-\varphi(k))\cdot\nabla_y g(x/R)h dtdx\ge 0.
\end{eqnarray}
Passing in (\ref{ad8}) to the limit as $R\to+\infty$ and taking into account the statement of Lemma~\ref{lemad2}, we obtain the relation
$$
\int_{\R_+}\left(\dashint_{\R^n} [|u-k|h_t+\sign(u-k)(\varphi(u)-\varphi(k))\cdot\nabla_x h]dx\right)dt\ge 0.
$$
In view of (\ref{ad0}) we can rewrite this relation as follows
\begin{equation}\label{ad9}
\int_{\R_+\times\B_n}[|v-k|\hat h_t+\sign(v-k)(\varphi(v)-\varphi(k))\cdot\nabla_x^y\hat h]dtdm(y)\ge 0,
\end{equation}
where $v=\hat u(t,y)$. Since $\hat h(t,y)$ is arbitrary nonnegative function from ${C^1_0(\R_+\times\B_n)}$, we claim that
entropy relation (\ref{ad7}) holds. Besides, again by (\ref{ad0})
$$
\int_{\B_n}|v(t,y)-v_0(y)|dm(y)=\dashint_{\R^n}|u(t,x)-u_0(x)|dx=N_1(u(t,\cdot)-u_0)\mathop{\to}_{t\to 0} 0
$$
and initial condition (\ref{adini}) in the sense of Definition~\ref{def2} is also satisfied.

To complete the proof of Theorem~\ref{adth1}, it only remains to establish the uniqueness. Suppose that $v_1=v_1(t,y)\in L^\infty(\R_+\times\B_n)$ be an e.s. of (\ref{adeq}), (\ref{adini}). We are going to demonstrate that
$v_1=v=\hat u(t,y)$. Recall that $u=u(t,x)\in C([0,+\infty),\B^1(\R^n))\cap L^\infty(\Pi)$ be an almost periodic e.s. to the problem (\ref{1}), (\ref{ini}) with the initial functions $u_0(x)$. Denote by $M_0=M(u_0)$ the minimal additive subgroup of $\R^n$ containing $Sp(u_0)$. By Theorem~\ref{th1} $Sp(u(t,\cdot))\subset M_0$ for all $t>0$. Let $\Phi_r(x)$, $r\in\N$ be the sequence of Bochner-Fej\'er kernels corresponding to a basis of $\Q$-linear subspace generated by $M_0$ (~or, the same, $Sp(u_0)$~). Let $\rho(s)\in C_0^\infty(\R)$ be a function such that $\supp\rho(s)\subset [0,1]$, $\rho(s)\ge 0$,
${\displaystyle\int_{-\infty}^{+\infty}\rho(s)ds=1}$. We introduce the approximate unity $\delta_\nu(s)=\nu\rho(\nu s)$, $\nu\in\N$, so that the sequence $\delta_\nu(s)$ converges as $\nu\to\infty$ to the Dirac $\delta$-measure weakly in $\D'(\R)$. We are going to adapt the Kruzhkov method of doubling variables to the pair of e.s. $v_1,v$. For that we apply (\ref{ad7}) for e.s. $v_1$ and
$k=v(s,z)$, $(s,z)\in\R_+\times\B_n$, to the test function $f(t,y)\delta_\nu(s-t)\widehat{\Phi_r}(y-z)$, where $f(t,y)\in C_0^1(\R_+\times\B_n)$, $f(t,y)\ge 0$. Integrating the result over $(s,z)$, we arrive at
\begin{eqnarray}\label{ad9a}
\int_{\R_+\times\B_n}\int_{\R_+\times\B_n}[|v_1(t,y)-v(s,z)|(f(t,y)\delta_\nu(s-t))_t\widehat{\Phi_r}(y-z)+\nonumber\\
\sign(v_1(t,y)-v(s,z))(\varphi(v_1(t,y))-\varphi(v(s,z)))\cdot\nabla_x^y(f(t,y)\widehat{\Phi_r}(y-z))\nonumber\\ \times\delta_\nu(s-t)]dtdm(y)dsdm(z)\ge 0.
\end{eqnarray}
In the same way, changing the places of variables $(t,y)$ and $(s,z)$ and e.s. $v_1(t,y)$ and $v(s,z)$, we derive from (\ref{ad7}) with $v=v(s,z)$, $k=v_1(t,y)$ that
\begin{eqnarray}\label{ad10}
\int_{\R_+\times\B_n}\int_{\R_+\times\B_n}[|v_1(t,y)-v(s,z)|(\delta_\nu(s-t))_sf(t,y)\widehat{\Phi_r}(y-z)+\nonumber\\
\sign(v_1(t,y)-v(s,z))(\varphi(v_1(t,y))-\varphi(v(s,z)))\cdot\nabla_x^z(\widehat{\Phi_r}(y-z))\nonumber\\ \times f(t,y)\delta_\nu(s-t)]dtdm(y)dsdm(z)\ge 0.
\end{eqnarray}
Putting (\ref{ad9a}) and (\ref{ad10}) together and taking into account the obvious identities
\begin{eqnarray*}
(f(t,y)\delta_\nu(s-t))_t+(\delta_\nu(s-t))_sf(t,y)=f_t(t,y)\delta_\nu(s-t), \\ \nabla_x^y(f(t,y)\widehat{\Phi_r}(y-z))+\nabla_x^z(\widehat{\Phi_r}(y-z))f(t,y)=(\nabla_x^y f(t,y))\widehat{\Phi_r}(y-z),
\end{eqnarray*}
we find that
\begin{eqnarray}\label{ad11}
\int_{\R_+\times\B_n}\int_{\R_+\times\B_n}[|v_1(t,y)-v(s,z)|f_t(t,y)+\nonumber\\
\sign(v_1(t,y)-v(s,z))(\varphi(v_1(t,y))-\varphi(v(s,z)))\cdot \nabla_x^y f(t,y)]\nonumber\\ \times\delta_\nu(s-t)\widehat{\Phi_r}(y-z)dtdm(y)dsdm(z)\ge 0.
\end{eqnarray}
First, we pass to the limit in (\ref{ad11}) as $\nu\to\infty$. We have the following inequalities
\begin{eqnarray}\label{adineq}
||v_1(t,y)-v(s,z)|-|v_1(t,y)-v(t,z)||\le |v(s,z)-v(t,z)|, \nonumber\\
|\sign(v_1(t,y)-v(s,z))(\varphi(v_1(t,y))-\varphi(v(s,z)))- \nonumber\\ \sign(v_1(t,y)-v(t,z))(\varphi(v_1(t,y))-\varphi(v(t,z)))| \le 2\omega(|v(s,z)-v(t,z)|),
\end{eqnarray}
where $\displaystyle \omega(r)=\max_{u,v\in [-M,M], |u-v|\le r} |\varphi(u)-\varphi(v)|$ is the continuity modulus of the vector $\varphi(u)$ on the segment $[-M,M]$, $M=\|v\|_\infty$. Let $\chi(t)\in C_0(\R_+)$ be a nonnegative function such that $\chi(t)=1$ if $(t,y)\in\supp f$.  By relations (\ref{Leb1}), (\ref{Leb2}), for a.e. $z\in\B_n$
\begin{eqnarray*}
\lim_{\nu\to\infty}\int_{\R_+\times\R_+} |v(s,z)-v(t,z)|\chi(t)\delta_\nu(s-t)dsdt=\\ \lim_{\nu\to\infty}\int_{\R_+\times\R_+} \omega(|v(s,z)-v(t,z)|)\chi(t)\delta_\nu(s-t)dsdt=0.
\end{eqnarray*}
This, together with (\ref{adineq}), implies that
\begin{eqnarray}\label{ad12}
\lim_{\nu\to\infty}\int_{\R_+\times\B_n}\int_{\R_+\times\B_n}[|v_1(t,y)-v(s,z)|f_t(t,y)+\nonumber\\
\sign(v_1(t,y)-v(s,z))(\varphi(v_1(t,y))-\varphi(v(s,z)))\cdot \nabla_x^y f(t,y)]\nonumber\\ \times\delta_\nu(s-t)\widehat{\Phi_r}(y-z)dtdm(y)dsdm(z)=\nonumber\\
\lim_{\nu\to\infty}\int_{\R_+\times\B_n}\int_{\R_+\times\B_n}[|v_1(t,y)-v(t,z)|f_t(t,y)+\nonumber\\
\sign(v_1(t,y)-v(t,z))(\varphi(v_1(t,y))-\varphi(v(t,z)))\cdot \nabla_x^y f(t,y)]\nonumber\\ \times\delta_\nu(s-t)\widehat{\Phi_r}(y-z)dtdm(y)dsdm(z)=\nonumber\\
\int_{\R_+\times\B_n\times\B_n}[|v_1(t,y)-v(t,z)|f_t(t,y)+\nonumber\\
\sign(v_1(t,y)-v(t,z))(\varphi(v_1(t,y))-\varphi(v(t,z)))\cdot \nabla_x^y f(t,y)]\nonumber\\ \times\widehat{\Phi_r}(y-z)dtdm(y)dm(z).
\end{eqnarray}
In view of (\ref{ad12}) it follows from (\ref{ad11}) that
\begin{eqnarray}\label{ad13}
\int_{\R_+\times\B_n\times\B_n}[|v_1(t,y)-v(t,z)|f_t(t,y)+\nonumber\\
\sign(v_1(t,y)-v(t,z))(\varphi(v_1(t,y))-\varphi(v(t,z)))\cdot \nabla_x^y f(t,y)]\nonumber\\ \times\widehat{\Phi_r}(y-z)dtdm(y)dm(z)\ge 0.
\end{eqnarray}
Passing in (\ref{ad13}) to the limit as $r\to\infty$, and taking into account the inequalities
\begin{eqnarray*}
||v_1(t,y)-v(t,z)|-|v_1(t,y)-v(t,y)||\le |v(t,y)-v(t,z)|, \\
|\sign(v_1(t,y)-v(t,z))(\varphi(v_1(t,y))-\varphi(v(t,z)))- \\ \sign(v_1(t,y)-v(t,y))(\varphi(v_1(t,y))-\varphi(v(t,y)))|\le 2\omega(|v(t,y)-v(t,z)|)
\end{eqnarray*}
together with the statements of Corollaries~\ref{corad1},~\ref{corad2} (notice that $Sp(v(t,\cdot))\subset M_0\subset H$), we arrive at
\begin{eqnarray}\label{ad14}
\int_{\R_+\times\B_n}[|v_1(t,y)-v(t,y)|f_t(t,y)+\nonumber\\
\sign(v_1(t,y)-v(t,y))(\varphi(v_1(t,y))-\varphi(v(t,y)))\cdot \nabla_x^y f(t,y)]dtdm(y)\ge 0.
\end{eqnarray}
Taking in (\ref{ad14}) the test functions $f=f(t)\in C_0^1(\R_+)$, $f(t)\ge 0$, we obtain that
$$
\int_0^{+\infty}\left(\int_{\B_n}|v_1(t,y)-v(t,y)|dm(y)\right)f'(t)dt\ge 0.
$$
This relation means that $I'(t)\le 0$ in $\D'(\R_+)$, where we denote
$$
I(t)=\int_{\B_n}|v_1(t,y)-v(t,y)|dm(y).
$$
Therefore, for a.e. $t>0$
\begin{equation}\label{ad15}
I(t)\le I(0)\doteq\esslim_{t\to 0+} I(t)=0.
\end{equation}
The latter equality follows from the estimate
\begin{eqnarray*}
I(t)=\int_{\B_n}|v_1(t,y)-v(t,y)|dm(y)\le\\ \int_{\B_n}|v_1(t,y)-v_0(y)|dm(y)+\int_{\B_n}|v(t,y)-v_0(y)|dm(y)
\end{eqnarray*}
and the initial requirement (\ref{adini}) for the e.s. $v,v_1$ (~in the sense of Definition~\ref{def2}~).
It now follows from (\ref{ad15}) that $I(t)=0$ for a.e. $t>0$, which implies that $v_1=v$ a.e. in $\R_+\times\B_n$
and completes the proof.
\end{proof}

\begin{remark}\label{adrem}
By Theorems~\ref{th1},~\ref{adth1}, the e.s. $v(t,y)$ of (\ref{adeq}), (\ref{adini}) belongs to the space $C([0,+\infty),L^1(\B_n,m))$, after possible correction on a set of null measure. Moreover, $Sp(v(t,\cdot))\subset M_0$ for all $t>0$, where $M_0$ is the additive subgroup of $\R^n$ generated by $Sp(v_0)$.
Besides, if $v_1(t,y),v_2(t,y)$ are e.s. of (\ref{adeq}), (\ref{adini}) with initial data $v_{01}(y)$, $v_{02}(y)$, respectively, then, as follows from the relation like (\ref{ad15}), for all $t>0$
$$
\int_{\B_n}|v_1(t,y)-v_2(t,y)|dm(y)\le \int_{\B_n}|v_{01}(y)-v_{02}(y)|dm(y)
$$
(~$L^1$-contraction property~). Obviously, this property implies the uniqueness of e.s.
\end{remark}

It readily follows from the results of Theorem~\ref{adth1} and Theorem~\ref{th2} the decay property for e.s. of
(\ref{adeq}), (\ref{adini}).

\begin{theorem}\label{adth2}
Let $M_0$ be an additive subgroup of $\R^n$. Then non-degeneracy condition (\ref{ND}) is necessary and sufficient for the decay property
$$
\esslim_{t\to+\infty}\int_{\B_n}|v(t,y)-C|dm(y)=0, \quad C=\int_{\B_n} v_0(y)dm(y)
$$
of every e.s. to (\ref{adeq}), (\ref{adini}) such that $Sp(v_0)\subset M_0$.
\end{theorem}

{\bf Acknowledgement.}
This research was carried out with the financial support of the Russian Foundation for Basic Research (grant no. 12-01-00230) and the Ministry of Education and Science of Russian Federation (in the framework of state task).

\end{document}